%% file: main.tex
\documentclass[12pt]{article}
\usepackage{amsmath,amssymb,amsthm, amsfonts,mathtools}

\usepackage{hyperref}
\usepackage{theoremref}
\usepackage{graphicx}
\usepackage{url}
\usepackage[mathlines]{lineno}
\usepackage{dsfont} 
\usepackage{color}
\definecolor{red}{rgb}{1,0,0}

\definecolor{blue}{rgb}{0,0,1}

\definecolor{green}{rgb}{0,.6,0}

\usepackage{tkz-euclide}
\usepackage{algpseudocode}
\usepackage{algorithm}

\usepackage{xcolor}
\colorlet{outline}{black}
\colorlet{pmu}{gray}

\usepackage[numbers]{natbib}
\input{macros.tex}

\title{Positive Semidefinite Initial Cost Product Throttling}
\author{Esther Conrad \thanks{Department of Mathematics, Iowa State University, Ames, IA 50011 (edconrad@iastate.edu)}}

\begin{document}
\maketitle

\begin{abstract} 
Product throttling answers the question of minimizing the product of the resources needed
to accomplish a task, and the time in which it takes to accomplish the task. In product throttling for positive semidefinite zero forcing,
task that we wish to accomplish is positive semidefinite zero forcing.
Positive semidefinite zero forcing is a game played on a graph $G$ that starts with a coloring
of the vertices as white and blue. At each step any vertex colored blue with a unique white
neighbor in a component of the graph formed by deleting the blue vertices from $G$ \textit{forces}
the color of the white neighbor to become blue. We give various results and bounds on the initial cost
product throttling number, including a lower bound of $1+\rad(G)$ and the initial cost product throttling number of a cycle. We also include a table with results on the initial cost and no initial cost product throttling number for various graph families. 
\end{abstract}

     \noi {\bf Keywords} positive semidefinite, zero forcing, throttling, propagation time
    
     \noi{\bf AMS subject classification} 05C69, 05C57, 05C85, 68R10, 05C50

\input{introduction.tex}

\input{general_bounds.tex}

      \input{graph_operations.tex}
    
      \input{special_graphs.tex}
      
      \bibliography{bib.bib}

\end{document}

%% file: macros.tex
\newcommand{\lb}{\left\{}
\newcommand{\rb}{\right\}}

\newcommand{\lc}{\left\lceil}
\newcommand{\rc}{\right\rceil}
\newcommand{\lp}{\left (}
\newcommand{\rp}{\right )}
\newcommand{\ls}{\left [}
\newcommand{\rs}{\right ]}

\newcommand{\ben}{\begin{enumerate}}
\newcommand{\een}{\end{enumerate}}

\newcommand{\noi}{\noindent}

\newcommand{\rad}{\operatorname{rad}}
\newcommand{\udot}{\dot{\cup}}
\newcommand{\dv}[1]{\frac{\mathrm{d}}{\mathrm{d}#1}}

\newcommand{\Sx}{S^\times}
\newcommand{\Ss}{S^*}
\newcommand{\F}{\mathcal{F}}
\newcommand{\Z}{\mathbb{Z}}

\newtheorem{thm}{Theorem}[section]
\newtheorem{cor}[thm]{Corollary}
\newtheorem{lem}[thm]{Lemma}
\newtheorem{prop}[thm]{Proposition}

\newtheorem{obs}[thm]{Observation}

\theoremstyle{definition}
\newtheorem{rem}[thm]{Remark}
\newtheorem{remark}[thm]{Remark}
\theoremstyle{definition}

\theoremstyle{definition}
\newtheorem{ex}[thm]{Example}

\newcommand{\psdzf}{PSD zero forcing }
\newcommand{\psd}{PSD }
\newcommand{\psdset}{PSD forcing set }

\newcommand{\psdth}{\operatorname{th_+}}
\newcommand{\psdthx}{\operatorname{th_+^\times}}
\newcommand{\psdthxs}[2]{\psdthx\lp#1;#2\rp}
\newcommand{\psdths}{\operatorname{th_+^*}}
\newcommand{\psdthss}[2]{\psdths\lp#1;#2\rp}
\newcommand{\psdpt}{\operatorname{pt_+}}
\newcommand{\psdpts}[2]{\psdpt\lp#1;#2\rp}
\newcommand{\psdz}{\operatorname{Z_+}}

\DeclareMathOperator{\rd}{rd}

\newcommand{\cop}{c}
\newcommand{\capt}{{\mathrm{capt}}}
\newcommand{\thcx}{\operatorname{th}_c^\times}
\newcommand{\thcs}{\operatorname{th}_c^*}

%% file: introduction.tex
\section{Introduction and Preliminaries}\label{sec:preliminaries}
Positive semidefinite zero forcing is a process on a graph where vertices have two possible colors - blue and white. The goal is to find an initial set of blue vertices so that at the end of the process, which is to repeat a PSD color change rule, all the vertices in the graph are blue. In 2013, Butler and Young began the study of \textit{(sum) throttling} in zero forcing \cite{butler2013throttling}, which is the problem of minimizing the sum of the number of blue vertices in a zero forcing set and the time it takes the set to color the entire graph blue. In 2019, Carlson et al. extended (sum) throttling to the PSD color change rule \cite{carlson2019throttling}.

In \cite{bonato2022optimizing}, Bonato et al. extended throttling to initial cost product throttling for Cops and Robbers, and in \cite{anderson2021product} Anderson et al. introduced the the terminology \textit{initial cost product throttling} and \textit{no initial cost product throttling} and extended these definitions to various other graph parameters including positive semidefinite zero forcing. Initial cost and no initial cost product throttling measure the relationship between the cost of initial blue vertices and the time it takes to color the entire graph via minimizing their product. The idea behind initial cost product throttling is to include the cost of placing the blue vertices as an extra time step. For example, the cost of placing a cop at an intersection. So, initial cost product throttling is the problem of minimizing the product of the size of the initial set of blue vertices and one plus the propagation time of this initial set. No initial cost product throttling on the other hand, does not take into account the cost of placement. So, the no initial cost product throttling problem is to minimize the product of the initial set and the propagation time of the initial set, but the solution that all vertices be blue initially is excluded. 

This paper extends the results for PSD product throttling in \cite{anderson2021product}, focusing primarily on initial cost product throttling.
In this section, we define PSD zero forcing, PSD propagation time, and PSD product and sum throttling, and present various basic graph definitions. In Section \ref{sec:general_bounds}  we present some general bounds for initial cost and no initial cost PSD product throttling including $\rad(G)+1$ as a lower bound for the initial cost PSD product throttling for any graph $G$. In Section \ref{sec:graph_operations}, we present some bounds for $G\Box H$ and edge operations. In Section \ref{sec:specialgraphs}, we establish the value of the product throttling number for a cycle and present a table with results for various families of graphs. 

\subsection{Positive semidefinite zero forcing, propagation, and throttling}
We begin by defining the \textit{positive semidefinite zero forcing color change rule} as in \cite{barioli2010zero}: 
Let $G$ be a graph and $S$ a set consisting of blue vertices. Let $W_1,\ldots,W_k$ be the sets of vertices of the components of $G - S$. Let $w\in W_i$. If $u\in S$ and $w$ is the only white neighbor of $u$ in $G[W_i\cup S]$, then change the color of $w$ to blue. 

Repeatedly applying the \psdzf color change rule until no more color changes can be made is called the \textit{\psdzf process}. The set of blue vertices at the end of the \psdzf process is called \textit{the derived set}. If the derived set is $V(G)$, the initial set $S$ is called a \textit{positive semidefinite zero forcing set} or a \textit{PSD forcing set}. The \textit{\psdzf number of $G$}, denoted $\psdz(G)$, is the minimum cardinality over $S$ such that $S$ is a \psdset of $G$.

Each time the color change rule is applied to all the currently blue vertices in a graph is a \textit{round}. If a white vertex $u$, is the unique white neighbor of a blue vertex $v$ in one component of $G-S$, we say that \textit{$v$ can force $u$} and $u$ turns blue in this round.
If $S\subseteq V(G)$, define $S^{[i]}$ to be the vertices that are blue after round $i$ of the \psdzf color change rule and $S^{(i)}$ to be the set of vertices that were changed to blue at the $i^{th}$ iteration of the \psdzf color change rule. That is 
\begin{align*}
    &S^{(0)} = S^{[0]} = S. \\
    &\text{For }i\geq 1,\\
    &\hspace{0.5cm}S^{(i)} = \left\{ w\in V(G)\setminus S^{[i-1]} : w \text{ can be forced by some $v$ given $S^{[i-1]}$}\right\}.\\
    &\hspace{0.5cm}S^{[i]} = S^{[i-1]}\cup S^{(i)}.
\end{align*}
For each $v\in V\lp G \rp $ define the \textit{round function}, $\rd\lp v \rp $, to be number of the round in which vertex $v$ is first colored blue. That is, $\rd\lp v \rp  = k$ for $v \in {S^{(k)}}$.

If a vertex $v$ is used to change the color of a vertex $u$ via the PSD zero forcing rule, we say that $v$ forces $u$, denoted as $v\to u$. Note that there may be more than one vertex that can force $u$, but only one vertex is chosen to force $u$. Let $S$ be a zero forcing set and construct the derived set, and let  the \textit{set of forces} be $\mathcal{F} = \{u\to v: \text{the force $u\to v$ was used to construct the derived set}\}$. Note that in order to construct the derived set, there is often a choice regarding which vertex will force a particular vertex. So a set of forces is not necessarily unique. 

Let $S$ be a PSD forcing set and $\F$ a set of forces for $S$, and let $E(\F) = \{uv: u \to v\in \F\}$. Then $\mathcal{T} = (V(G), E(\F))$ is a union of trees. For every vertex $v\in S$, $\mathcal{T}$ contains a tree $T_v$ such that $v\in T_v$, and for any distinct vertices $u,v\in S$, $V(T_v)\cap V(T_u) = \emptyset.$ A vertex $v\in S$ is the root of a tree $T_v$. The graph $\mathcal{T}$ is called a \textit{PSD forcing tree cover}.

In \cite{warnberg2016positive}, Warnberg defined the \textit{positive semidefinite propagation time} of $S$ in $G$ as the smallest $p$ such that $S^{[p]} = V(G)$, and denoted it as $\psdpts{G}{S}$. If $S$ is not a PSD forcing set, then $\psdpts{G}{S}= \infty$. For $k\in \Z^+$, $\psdpt(G,k) = \min\{\psdpts{G}{S}: |S| = k\}$ and \textit{the \psd propagation time of $G$} is $\psdpt(G) = \psdpt(G,\psdz(G)).$

 Positive semidefinite throttling was introduced in \cite{carlson2019throttling} as an extension to throttling for zero forcing \cite{butler2013throttling}. Let $S$ be an initial set of blue vertices. Define the \textit{positive semidefinite throttling number of $S$ in $G$ } as $\psdth(G;S) = |S| + \psdpts{G}{S}$. For $k\in \Z^+$, $\psdth(G,k) = \min\{\psdth(G;S): |S| = k\}$ and the \textit{positive semidefinite throttling number of $G$} is $\psdth(G) = \min\{\psdth(G,k): k\geq \psdz(G)\}.$ The \textit{initial cost positive semidefinite product throttling number of $S$ in $G$ } is $\psdthxs{G}{S} = |S|(1+\psdpts{G}{S})$. For $k\in \Z^+$, $\psdthx(G,k) = \min\{\psdthxs{G}{S}: |S| = k\}$ and the \textit{initial cost positive semidefinite product throttling number of $G$} is 
 $$\psdthx(G) = \min\{\psdthx(G,k): k\geq \psdz(G)\}.$$ 
 
 The \textit{no initial cost positive semidefinite product throttling number of $S$ in $G$ } is $\psdthss{G}{S} = |S|(\psdpts{G}{S})$. For $k\in\Z^+$, $\psdths(G,k) = \min\{\psdthss{G}{S}: |S| = k\}$ and the \textit{no initial cost positive semidefinite product throttling number of $G$} is $\psdths(G) = \min\{\psdths(G,k): \psdz(G)\leq k< V(G)\}$. 

\textit{Cops and Robbers} is a two player game on a graph. The first player places and moves a collection of cops and the second player places and moves a single robber. The players exchange turns when moving the cops and robber. The goal of the first player is for a cop to capture the robber by occupying the same vertex and the goal of the second player is for the robber to evade capture. At each round, each cop is allowed to move to an adjacent vertex or stay in the same place, followed by the robber who is allowed to move to an adjacent vertex or stay in the same place. The \textit{cop number} of the a graph $G$ is the minimum number of cops needed to catch a robber and is denoted $\cop(G)$. The \textit{capture time} of $G$ of an initial set $S$ is the minimum number of rounds it takes cops placed on the vertices in $S$ to capture a robber assuming the players are using optimal strategies and is denoted $\capt(G;S)$. If the set $S$ cannot capture the robber then $\capt(G;S) = \infty$ \cite{copsandrobbers}. The $k$-capture time,
$\capt_k(G) = \min\{\capt(G;S) : |S| =k\}$. Let $G$ be a graph. The initial cost product throttling for cops and robbers of an initial set $S$ is $\thcx\lp  G;S\rp = |S|(1+\capt(G;S))$. The initial cost product throttling number of $G$ for cops and robbers is $\thcx\lp  G\rp = \min\{k(1+\capt_k(G)) : k \geq \cop(G) \}$.

The next two results are used for examples throughout this paper.
 \begin{thm}{\textrm{\cite{anderson2021product}}}\thlabel{prop:cops_tree_cycle} Let $G$ be a tree or cycle. Then $\psdpts{G}{S} = \capt\lp  G;S\rp$. Consequently, for any $S\subseteq V\lp  G\rp$, $\psdthxs{G}{S} = \thcx\lp  G;S\rp$ and $\psdthss{G}{S} = \thcs\lp  G;S\rp$, and thus $\psdthx\lp  G\rp = \thcx\lp  G\rp$ and $\psdths\lp  G\rp = \thcs\lp  G\rp$.
 \end{thm}
 \begin{cor}{\rm{\cite{anderson2021product, bonato2022optimizing}}}\thlabel{cor:trees} Let $T$ be a tree. Then $\psdthx\lp  T\rp = 1 +\rad\lp  T\rp$. In particular, $\psdthx\lp  P_n\rp = 1 + \lc\frac{n-1}{2}\rc$.
 \end{cor}

\subsection{Graph theory}
    A \textit{graph} $G = (V,E)$ is a set of \textit{vertices}, $V = V(G)$ and a set of \textit{edges}, $E = E(G)$. An edge has the form $\{u,v\}$, for some $u\neq v, \in V(G)$. The edge $\{u,v\}$ is often written as $uv$. For vertices $u,v$ in a graph, a $u-v$ \textit{path} is a sequence of vertices and edges $u = v_1,e_1,\ldots,v_\ell,e_\ell,v_{\ell+1}=v$. Each vertex $v_i$ is distinct and $e_i = \{v_i,v_{i+1}\}$ for each $i\leq \ell$. A \textit{cycle} is a sequence of vertices and edges $u = v_1,e_1,\ldots,v_\ell,e_\ell,v_{\ell+1}=u$. Each vertex $v_i\neq v_j$ for $1\leq i,j\leq \ell$, and $e_i = \{v_i,v_{i+1}\}$ for each $i\leq \ell$. The number of edges in a path or cycle is the \textit{length} of the path or cycle. A graph $G$ is \textit{connected} if there is a path from any vertex to any other vertex. The \textit{distance} between the vertices $u$ and $v$ is the minimum length of a $u-v$ path, denoted $d(u,v)$. The \textit{eccentricity} of a vertex $v$ is $e(v) = \max\{d(u,v): u \in V(G) \setminus \{v\}\}.$ The \textit{radius of $G$} is $\rad{G} = \min\{e(v):v\in V(G)\}$. Given a set $S$ of vertices and a vertex $u$, define the \textit{distance between $S$ and $u$} as $d(S,u)= \min\{d(v,u):v\in S\}$, and define the \textit{eccentricity of $S$} to be $e(S) = \max\{d(S,u): u \in V\}$. The \textit{$k$-radius} is $\rad_k(G) = \min\{e(S): |S| = k$ and $S\subseteq V\}$.

 We say that vertices $u$ and $v$ are \textit{neighbors} if $uv \in E$. The \textit{neighborhood} of a vertex $v$, denoted $N(v)$ is the set of the neighbors of $v$. The \textit{closed neighborhood} of $v$ is $N[v] = N(v)\cup\{v\}$. The \emph{degree} of $v$ is the number of neighbors of $v$, denoted $\deg{v}{}$. When the graph $G$ is not clear from context, we use $\deg{v}{G}$ and $N_G(v)$ as needed. The\textit{ maximum degree} of a graph is $\Delta(G) = \max_{v\in V} \deg{v}{}$. A set of vertices $S$ is an \textit{independent set} if for any $u,v \in S$, $uv \not\in E(G)$. We denote the cardinality of the smallest independent set by $\alpha(G)$.
 
 A \textit{subgraph $H$} of a graph $G$ is a graph such that $V(H) \subseteq V(G)$ and $E(H) \subseteq E(G).$ An \textit{induced subgraph} $H$ of $G$ is a subgraph of $G$ such that $E(H) = \{uv: u,v\in V(H)$ and $uv \in E(G)\}$. If $W\subseteq V(G)$, $G[W]$ denotes the induced subgraph with vertex set $W$.

 A \textit{tree} is a connected graph with no cycles. A \textit{complete bipartite graph} $K_{a,b}$ is the graph with vertex set $V = X\cup Y$ such that $|X|=a$, $|Y| = b$, $X\cap Y = \emptyset$, and the edge set is $E = \{xy: x\in X, y\in Y\}$. 
 
 We refer the reader to \textit{Graph Theory} by Diestel \cite{Diestelbook} as a reference for any additional graph terminology.

%% file: general_bounds.tex
\section{General Bounds}\label{sec:general_bounds}
We begin by citing some general statements from \cite{anderson2021product}.
\begin{obs}{\rm{\cite{anderson2021product}}} \thlabel{obs:PSDextreme}
Let $G$ be a graph of order $n$.
\begin{enumerate}
    \item If $\psdz(G) \geq \frac{n}{2}$, then $\psdthx(G) = n$.
    \item $\psdthx(G) \geq th_c^\times(G)$.
\end{enumerate}
\end{obs}
The extreme low values in the remark below can easily be found by factoring $\psdthx(G)$.
\begin{remark}\cite{anderson2021product}\thlabel{rem:lowbounds}
Let $G$ be a graph.
\begin{enumerate}
    \item $\psdthx(G) = 1$ if and only if $G= K_1$,
    \item $\psdthx(G) = 2$ if and only if $G= 2K_1$ or $G$ is a star,
    \item $\psdthx(G) = 3$ if and only if $G$ satisfies exactly one of the following conditions:
    \begin{enumerate}
        \item $G = K_3$, $G = K_2\udot K_1$, or $G = 3K_1$,
        \item $G$ is a tree of radius $2$,
    \end{enumerate}
    \item $\psdthx(G) = 4$ if and only if $G$ satisfies at least one of the following conditions:
    \begin{enumerate}
        \item $G = K_4$, $G = K_3\udot K_1$, $G = K_2\udot 2K_1$, or $G = 4K_1$,
        \item $\psdpt(G,2) = 1$ and $\psdpt(G,1)>3$,
        \item $\psdz(G) = 1$ and $\psdpt(G)- 3$.
    \end{enumerate}
\end{enumerate}
\end{remark}

If the positive semidefinite zero forcing propagation time of a graph is $1$, the propagation time of any \psdset is $0$ or $1$. So, the result for product throttling (initial cost and no initial cost) is immediate.
\begin{obs}\thlabel{obs:PSDproptimeisone}
    Let $G$ be a graph. If $\psdpt(G) = 1$, then $\psdthx(G) = \min\{n, 2\psdz(G)\}$ and $\psdths(G) = \psdz(G)$.
\end{obs}

In \cite{anderson2021product}, Anderson et al. remarked that we can use the $k$-radius of a graph to find a bound for the product throttling: $pt_+(G,k) \geq \rad_k(G)$ and $\psdthx(G) \geq \min_{Z_+(G)\leq k\leq n}k(1+\rad_k(G)).$ After establishing a lemma, we improve this to a tight bound.

\begin{lem}\thlabel{lem:PSDrad}
Let $G$ be a connected graph, $S$ a PSD forcing set of $G$, and $\mathcal{T}$ a PSD forcing tree covering of $S$. Then 
$$|S|-1 +|S|\max\{\rad(T):T\in \mathcal{T}\} \geq \rad(G).$$
\end{lem}
\begin{proof}
Let $X$ be a spanning tree of $G$ such that $\dot\bigcup_{T\in \mathcal{T}}T$ is a subgraph of $X$. 

Let $X'$ be the tree obtained from $X$ by contracting each subgraph $T\in \mathcal{T}$. Since each vertex in $S$ forms a tree in $\mathcal{T}$, $X'$ has $|S|$ vertices and $|S|-1$ edges. So, we have the following:
\begin{align*}
    \rad(G)&\leq \rad(X) 
     \leq  \rad(X') + \sum_{T\in \mathcal{T}}(\rad(T))
    \leq |S|-1 +|S|\max\{\rad(T):T\in \mathcal{T}\}. \qedhere
\end{align*}
\end{proof}

\begin{thm}\thlabel{thm:PSDrad}
Let $G$ be a connected graph. Then $\psdthx(G)\geq 1+\rad(G)$ with equality whenever $G$ is a tree. 
\end{thm}
\begin{proof}
Suppose that $S$ is a \psdset of $G$. Consider a PSD tree covering $\mathcal{T}$ of $S$. Observe that $\psdpts{G}{S} \geq \max\{\rad(T):T\in \mathcal{T}\}$. From this observation and \thref{lem:PSDrad} we have the following: 
\begin{align*}
    \psdthxs{G}{S} = |S|(1+\psdpts{G}{S})
    \geq |S|(1+\max\{\rad(T):T\in \mathcal{T}\})
    \geq 1+\rad(G).
\end{align*}

It is shown in \cite{anderson2021product} that $\psdthx(T) = 1 + \rad(T)$ for every tree $T$ (see also \thref{cor:trees} and \thref{prop:cops_tree_cycle}).
\end{proof}
The following is an example of a graph that is not a tree that achieves equality for the bound in \thref{thm:PSDrad}.
\begin{ex}
For $a\geq 1$, $1+\rad(C_{4a+6}) = 2a+4 = \psdthx(C_{4a+6})$ by \thref{thm:cycles} below.
\end{ex}

We wish to use \thref{lem:PSDrad} to develop a similar bound for no initial cost throttling. However, the immediate bound that is attainable is not useful.

\begin{rem} Let $G$ be a connected graph and $k\geq 1$. By \thref{lem:PSDrad}, $k-1 +k\psdpt(G,k) \geq \rad (G)$. So, $\psdths(G,k)\geq \rad(G) - k +1$. So, $$\psdths(G)\geq \min\{\rad(G) - k +1 : \psdz(G)\leq k < |V(G)|\}.$$
However, this bound is not useful, because $\rad(G) - (n-1) + 1 \leq 0$ for $n\geq 3$.
\end{rem}

One may also use an independent set of vertices to find an upper bound, because $\psdpts{G}{A} = 1$ for an independent set $A$. This implies $\psdth(G) \leq |V(G)|-\alpha(G) +1$ \cite{carlson2019throttling}, $\psdths(G)\leq |V(G)| - \alpha(G)$ \cite{anderson2021product}, and the next result.
\begin{prop}\thlabel{prop:PSDalpha}
    Let $G$ be a connected graph of order $n$. Then 
    $\psdthx(G) \leq 2(n-\alpha(G))$. This bound is tight.
    \end{prop}
    \begin{proof}
    This bound is immediate. To see that this is tight, consider the complete bipartite graph $K_{n,m}$. Note that $\psdz(K_{n,m}) = \min\{n,m\}$ and $\psdpt(K_{n,m}) = 1$. So, by \thref{obs:PSDproptimeisone}, $\psdthx(K_{n,m}) = 2(\min\{n,m\}) = 2(n-\alpha(K_{n,m}))$.
    \end{proof}

    The next result describes relationships between various parameters involving sets $\Sx$ and $\Ss$ that realize $\psdthx(G)$ and $\psdths(G)$.
\begin{prop}
    
     Let $G$ be a connected graph, $\Ss,\Sx\subseteq V(G)$ such that $\psdthx(G) < |V(G)|$, $\psdthx(G)  = \psdthxs{G}{\Sx}$, and $\psdths(G) = \psdthss{G}{\Ss}$. Then 
    \begin{enumerate}
        \item $\psdpts{G}{\Ss} = \psdpt(G,|\Ss|)$ and $\psdpts{G}{\Sx} = \psdpt(G,|\Sx|)$,
        \item $|\Ss| \geq |\Sx|$,
        \item $\psdpts{G}{\Ss} \leq \psdpts{G}{\Sx} $,
        \item If $|\Ss| = |\Sx|,$ then $\psdths(G) =\psdthss{G}{\Sx}$ and $\psdthx(G) =\psdthxs{G}{\Ss}$.
    \end{enumerate}
\end{prop}   
\begin{proof}
    We first prove (1). Since 
    \begin{align*}
        \psdthx(G) \leq \psdthx(G,|\Sx|) 
                 = |\Sx|(1+\psdpt(G,|\Sx|))
                \leq |\Sx|(1+\psdpts{G}{\Sx}) 
                = \psdthx(G),
    \end{align*}
it must be that $\psdpt(G,|\Sx|)) =  \psdpts{G}{\Sx}$. Similarly, $\psdpt(G,|\Ss|)) =  \psdpts{G}{\Ss}$.\\

We next prove (2). Observe the following:
\begin{align*}
    |\Sx|+ |\Sx|\psdpts{G}{\Sx} & = \psdthxs{G}{\Sx}\\
    &\leq \psdthxs{G}{\Ss} \\
    &=|\Ss|+ |\Ss|\psdpts{G}{\Ss}\\
    &=|\Ss|+ \psdthss{G}{\Ss}\\
    &\leq |\Ss|+ \psdthss{G}{\Sx}\\
    &= |\Ss|+ |\Sx|\psdpts{G}{\Sx}.
\end{align*}

Thus, $|\Sx| \leq |\Ss|$. \\

We next prove (3). Since $|\Ss| \geq |\Sx|$, and by (1),
 $$\psdpts{G}{\Ss} = \psdpt(G,|\Ss|) \leq \psdpt(G|\Sx|) = \psdpts{G}{\Sx}.$$

To prove (4), we have that 
\begin{align*}
    \psdthx(G) &= \psdpts{G}{\Sx}\\
     &= |\Sx|(1+\psdpts{G}{\Sx}) \\
     &= |\Sx|(1+\psdpt(G,|\Sx|))\\
     &=  |\Ss|(1+\psdpt(G,|\Ss|))\\
     &=|\Ss|(1+\psdpts{G}{\Ss})\\
     &= \psdthxs{G}{\Ss}.
\end{align*}
Similarly $\psdths(G)= \psdthss{G}{\Sx}$.
\end{proof}

The following is a result from \cite{carlson2019throttling} that was used for sum throttling.

\begin{lem}{\rm{\cite{carlson2019throttling}}} \thlabel{lem:PSDdelta_pt}
Suppose $G$ is a graph of order $n$ and $S$ is a positive semidefinite zero forcing set of $G$. Then,

$$n\leq \begin{cases}
|S|(1+2\psdpt(G;S)) & \text{if } \Delta(G)=2 \\
 |S|\left(1+\frac{\Delta(G)(\Delta(G)-1)^{\psdpt(G;S)}- \Delta(G)}{\Delta(G)-2} \right) & \text{if } \Delta(G)>2.
\end{cases}
$$
\end{lem}
We use this lemma to prove the next two results which are analogous to \cite[Proposition 2.5 and Theorem 2.6]{carlson2019throttling}.
\begin{prop}\thlabel{prop:PSDdelta_throttling}
Let $\Delta(G) = 2$. Then $$\psdthx(G) \geq \left\lceil \frac{1}{2}(Z_+(G) +n)\right\rceil$$
and this bound is tight.
\end{prop}

\begin{proof}
For a PSD zero forcing set $S$, let $s = |S|$ and $p = \psdpt(G,s)$. We wish to minimize $s(1+p)$ and $p$ is subject to $n\leq s(1+2p)$ by \thref{lem:PSDdelta_pt}. So, we have that $p \geq \frac{1}{2}\left(\frac{n}{s} -1 \right)$. Thus, 
\begin{align*}
    s(1+p) & \geq s\left(1 + \frac{n}{2s} - \frac{1}{2}\right)
     = \frac{s+n}{2}
    \geq \frac{\psdz(G) +n}{2}.
\end{align*}
\noindent Therefore, 
$$\psdthx(G) \geq \left\lceil \frac{1}{2}(Z_+(G) +n)\right\rceil.$$
\noindent\thref{cor:trees} shows that this bound is tight for paths.
\end{proof}
Notice that the \thref{prop:PSDdelta_throttling} is different than Proposition 2.5 in \cite{carlson2019throttling} which showed that $\psdth(G) \geq \lc\sqrt{2n}-\frac{1}{2}\rc$ for a graph $G$ of order $n$.
The next theorem is analogous to Theorem 2.6 in \cite{carlson2019throttling} where if $\Delta(G)\geq 3$, then,
$\psdth(G)\geq \left\lceil 1+ \log_{\Delta(G)-1}\left[\frac{(\Delta(G) -2)n +2 \psdz(G)}{\Delta(G)\psdz(G)}\right]\right\rceil$.
\begin{thm}
Let $\Delta(G)\geq 3$. Then, 
$$\psdthx(G)\geq \left\lceil \psdz(G) \left(1+ \log_{\Delta(G)-1}\left[\frac{(\Delta(G) -2)n +2 \psdz(G)}{\Delta(G)\psdz(G)}\right]\right)\right\rceil.$$
\end{thm}
\begin{proof}
By \thref{lem:PSDdelta_pt} we have that 
$n \geq |S|\left(1+\frac{\Delta(G)(\Delta(G)-1)^{\psdpt(G;S)}- \Delta(G)}{\Delta(G)-2}\right)$. Let $s,p, \Delta$ be $|S|, \psdpt(G;S),$ and $\Delta(G)$ respectively. 
Then,
$p \geq \frac{\ln\ls\frac{(\Delta -2)n +2s}{\Delta s}\rs}{\ln\lp\Delta -1\rp} $.

Define $p(s) = \frac{\ln \ls\frac{(\Delta -2)n +2s}{\Delta s}\rs}{\ln\lp\Delta -1\rp}$. Observe that $s(1+p) \geq s(1+p(s))$. In the proof for Theorem 2.6 in \cite{carlson2019throttling}, it is established that $\dv{s}(s+p) > 0$ when $s\geq$ and $\Delta \geq 4$, or when $s\geq 2$ and $\Delta =3$. So, under these conditions, 
$\dv{s}(s(1+p(s)) = 1 + s\dv{s}p(s) +p(s) \geq 1 + \dv{s}p(s) = \dv{s}(s+p(s)) >0.$

So, $s(1+p(s))$ is an increasing function of $s$ when $\Delta = 3$ and $s\geq 2$. So, when $\Delta = 3$, $\min\{s(1+p(s)): s\geq 1\}= \min\{1+p(1), 2(1+p(2))\}$. Now, in the proof for Theorem 2.6 in \cite{carlson2019throttling} it is shown that $\min\{1+p(1), 2+p(2))\} = 1 +p(1)$. And since $2+p(2) \leq 2(1+p(2))$, $\min\{1+p(1), 2(1+p(2))\} = 1+p(1)$. 

Therefore, for $\Delta \geq 3,$ 
$$\psdthx(G)\geq \left\lceil \psdz(G) \left(1+ \log_{\Delta(G)-1}\left[\frac{(\Delta(G) -2)n +2 \psdz(G)}{\Delta(G)\psdz(G)}\right]\right)\right\rceil. \qedhere$$
\end{proof}

In \cite{carlson2021various}, Carlson and Krischgau determined the following characterization for throttling with a finite family of finite induced subgraphs. 
\begin{thm}[\cite{carlson2021various}, Theorem 4.7]
Let $k$ be a non-negative integer and suppose $R$ is either the standard or PSD color change rule. The set of graphs
$G$ such that $\operatorname{th}_R(G) \geq |V(G)| - k$ and $|V(G)| \geq k$ is characterized by a finite family of forbidden induced subgraphs.
\end{thm}
Unfortunately, the immediate analog of characterization with a finite family of forbidden induced subgraphs is not true for initial value product throttling. 
\begin{prop}
Let $k$ be a positive integer, the set of graphs $G$ such that $\psdthx(G) \geq |V(G)|-k$ cannot be characterized by a finite family of forbidden induced subgraphs. 
\end{prop}
\begin{proof}
Let $H$ be a graph, and suppose $r\geq |V(H)|+2$. Let $G$ be a subgraph of $H \vee K_r$ that contains $H$ and $K_r$ as subgraphs and at least one additional edge (see Figure \ref{fig:PSDforbidden}). Then $\psdz(G) \geq r-1 \geq \frac{|V(G)|}{2}$, which by \thref{obs:PSDextreme} implies that $th^\times_+(G) = |V(G)|> |V(G)|-k$. Since $H$ is an arbitrary graph, we cannot forbid a finite family of induced graphs. 
\end{proof}
\begin{figure}[h]
    \centering
\begin{tikzpicture}
    \draw (0,0) rectangle (2,2) node[pos=.5] {$H$};
    \draw (4,1) circle (1cm) node {$K_r$};
    \draw (2,1)-- (3,1);
\end{tikzpicture}
    \caption{$G$}
     \label{fig:PSDforbidden}
\end{figure}

%% file: graph_operations.tex
\section{Graph Operations}\label{sec:graph_operations}
We now find bounds for the positive semidefinite initial value product throtlling number of different graph operations. 

Let $G$ and $G'$ be graphs. Define $G\Box G'$ to be the graph such that $V( G\Box G' )  = V( G ) \times V( G ') $ and $E( G\Box G' )  = \{( x,x' ) ( y,y' )  : (x = y$ and $x'y'\in E\lp G' \rp  ) $ or $( x' = y'$ and $xy\in E \lp G \rp )  \}$. For $S\subseteq V( G\Box G' ) $, define $S_G=\{x\in V( G ) :( x,x' ) \in S$ for some $x'\in V( G' ) \}$.

\thref{prop:PSDptProduct}, \thref{prop:PSDptProduct2}, and \thref{thm:PSDupperboundProduct} state bounds for $\psdthx\lp G\Box H\rp$. Since the proofs for these bounds are similar to their analogs found in \cite{anderson2020pdproduct}, their proofs are ommited.
\begin{prop}\thlabel{prop:PSDptProduct}
Let $G$ and $G'$ be graphs and $S$ be a \psdset for $G\Box G'$, then $S_G$ is a \psdset of $G$ and $\psdpts{G}{S_G}\leq \psdpts{G\Box G'}{S}$. This inequality is tight.
\end{prop}
\begin{proof}
For tightness, consider a grid $P_n\Box P_m$, where $V(P_n) = \lb 1,\ldots,n\rb$, $V(P_m) = \lb 1,\ldots,m\rb$, and $n\leq m$. Let $S = \lb\lp\lc \frac{n}{2}\rc, 1\rp,\ldots, \lp\lc \frac{n}{2}\rc, m\rp \rb$. So, $S_{P_n} = \lb\lc \frac{n}{2}\rc \rb$, and  $\psdpts{P_n\Box P_m}{S} = \lc \frac{n-1}{2}\rc  =\psdpts{P_n}{S_{P_n}}$.
\end{proof}

\begin{prop}\thlabel{prop:PSDptProduct2}
Let $G$ and $G'$ be graphs then $\psdthx\lp G\Box G' \rp  \geq \psdthx\lp G \rp $ and $\psdthx\lp G\Box G' \rp \geq \psdthx\lp G' \rp $. 
\end{prop}
Other than the trivial example $G\Box K_1 = G$, we do not have an example for tightness in the previous bound. 

\begin{thm}\thlabel{thm:PSDupperboundProduct}
Let $G$ and $G'$ be graphs, then $\psdthx\lp G\Box G' \rp  \leq \psdthx\lp G \rp |V\lp G' \rp |$ and  $\psdthx\lp G\Box G' \rp  \leq \psdthx\lp G' \rp |V\lp G \rp |.$ And this bound is tight. 
\end{thm}
\begin{proof}
To see that this bound is tight, consider the grid $P_2\Box P_4$. Observe that $\psdthx(P_4) = 3$, and $\psdthx(P_2\Box P_4) = 6$. Note that there $\psdz(P_2\Box P_4) = 2$. In order to force, these two vertices must be adjacent to each other, and the optimal placement gives $\psdpt(P_2\Box P_4) = 2$. So, $\psdthx(P_2\Box P_4)\leq 6$. And if $S$ is a set of $3$ vertices, $\psdpt(P_2\Box P_4;S) \geq 2$. So, it must be that $\psdthx(P_2\Box P_4)= 6$.
\end{proof}


Let $G$ be a graph and $e= \{x,y\}$ an edge in $G$. We use $G_e$ to denote the graph $G$ with $e$ subdivided, and $G-e = (V(G),E(G)\setminus\{e\})$. 
 In \cite[Theorem 5.4]{ekstrand2013positive} it is shown that $\psdz(G) = \psdz(G_e)$ and the proof shows that any \psdset $S$ for $G$ is also a \psdset for $G_e$.  It is proved in \cite[Theorem 10.37]{book} that $\psdpt(G_e) \leq \psdpt(G)+1$ by using the same zero forcing set in $G$ and $G_e$. Since the proof holds no matter what PSD zero forcing set is used, we have that $\psdpts{G_e}{S} \leq \psdpts{G}{S}+1.$ So, we have the next lemma. 

\begin{lem}{\thlabel{lem:subdivision}}
Let $e = \{u,w\}$ be an edge of a graph $G$, $S$ a \psdset of $G$. Then $S$ is a \psdset of $G_e$ and 
$\psdpts{G_e}{S} \leq \psdpts{G}{S}+1.$
\end{lem}

Next, we make use of the previous lemma to present a tight bound for edge subdivision. 
\begin{prop}\thlabel{prop:PSDsubdivision} Let $e$ be an edge of a graph $G$, and suppose $\psdthx(G) = \psdthx(G;S)$. Then
\begin{align*}
    \psdthx(G_e)&\leq \min\{|S|(2+ \psdpt(G;S)), (|S|+1)(1+ \psdpt(G;S))\}\\
& = \min\{\psdthx(G) +|S|, \psdthx(G) + \psdpt(G;S) +1\}.
\end{align*}
This bound is tight.
\end{prop}
\begin{proof}
By \thref{lem:subdivision} $S$ is a \psdset of $G_e$ and $\psdpts{G_e}{S}\leq \psdpts{G}{S} +1$. So $\psdthx(G_e)\leq |S|(2+ \psdpt(G;S))$.

Let $S' = S\cup\{v_e\}$. The proof that $\psdth(G)\leq \psdth(G_e)$ in \cite{book}, shows that $\psdpt(G_e,S')\leq \psdpt(G;S)$. Thus, $\psdthx{G_e} \leq \psdthxs{G_e}{S} = |S'|(1+\psdpts{G_e}{S'} )\leq (|S|+1)(1+\psdpts{G}{S}).$

Note that the statement can be restated as $\psdthx(G_e)\leq |S|(2+ \psdpt(G;S))$ and $\psdthx(G_e)\leq (|S|+1)(1+ \psdpt(G;S))$. We show that both of these bounds are tight. For the first inequality, consider the tree with $7$ vertices in Figure \ref{fig:operation}. By \thref{rem:lowbounds}, $\psdthx(G)= 3$, and if we subdivide an edge adjacent to a leaf, $\psdthx(G_e)= 4$. For the second inequality, consider let $G = K_n$ for $n\geq 4$ and $e$ be any edge in $G$. Observe that $\psdz(G_e) = |V(G_e)|- 2 \geq \frac{|V(G_e)|}{2}$, and $\psdthx(G_e) = |V(G_e)| = |V(G)|+1 = \psdthx(G) + 1$.  
\end{proof}
The next operation is edge deletion.   
\begin{prop} Let $G$ be a graph and $e$ be an edge of $G$, and suppose $\psdthx(G) = \psdthx(G;S)$. Then
$$\psdthx(G-e)\leq (|S|+1)(1+ \psdpt(G;S)) = \psdthx(G) + 1 + \psdpt(G;S).$$
\end{prop}
\begin{proof}
Suppose $e = \{u,v\}$. Construct a set of forces $\mathcal{F}$ for $S$ in $G$. Suppose without loss of generality that the force in which $u$ is colored blue appears before $v$. Let $S' = S \cup \{v\}$. We create a set of forces $\mathcal{F}'$ for $S'$ in $G$ by simply removing the force that in which $v$ is colored blue. Thus, $\psdpts{G}{S'}\leq \psdpts{G}{S}$ and $\psdthx(G-e)\leq (|S|+1)(1+ \psdpt(G;S)) = \psdthx(G) + 1 + \psdpt(G;S)$.
To see that this is tight, consider the tree with $7$ vertices in Figure \ref{fig:operation}. Let $e = uv$ be an edge such that $v$ is a leaf in $G$. Then $\psdthx(G)= 3$ (\thref{rem:lowbounds}) and $\psdthx(G-e)= 6$.
\end{proof}
\begin{figure}[h!]
    \centering
    \begin{tabular}{ccc}
        \begin{tikzpicture}[shorten >=8pt, shorten <=8pt, scale=0.99]
        \node[circle, draw=black, fill=pmu] at (0,0) {$ $}; 
        \node[circle, draw=black, fill=white] at (-1,-1) {$ $}; 
        \node[circle, draw=black, fill=white,label=left:{\small $u$}] at (1,-1) {$ $}; 
        \node[circle, draw=black, fill=white] at (-1.5,-2) {$ $}; 
        \node[circle, draw=black, fill=white] at (-.5,-2) {$ $}; 
        \node[circle, draw=black, fill=white] at (1.5,-2) {$ $}; 
        \node[circle, draw=black, fill=white, label=below:{\small $v$}] at (.5,-2) {$ $}; 
         \draw[-, thick] (0,0) -- (-1,-1);
         \draw[-, thick] (0,0) -- (1,-1);
         \draw[-, thick] (-1,-1) -- (-1.5,-2);
         \draw[-, thick] (-1,-1) -- (-.5,-2);
         \draw[-, thick] (1,-1) -- (1.5,-2);
         \draw[-, thick] (1,-1) -- (.5,-2);

        \end{tikzpicture}
        &
        \begin{tikzpicture}[shorten >=8pt, shorten <=8pt, scale=0.99]
            \node[circle, draw=black, fill=pmu] at (0,0) {$ $}; 
            \node[circle, draw=black, fill=white] at (-1,-1) {$ $}; 
            \node[circle, draw=black, fill=white,label=left:{\small $u$}] at (1,-1) {$ $}; 
            \node[circle, draw=black, fill=white] at (-1.5,-2) {$ $}; 
            \node[circle, draw=black, fill=white] at (-.5,-2) {$ $}; 
            \node[circle, draw=black, fill=white] at (1.5,-2) {$ $}; 
            \node[circle, draw=black, fill=white] at (.5,-2) {$ $}; 
            \node[circle, draw=black, fill=white,,label=below:{\small $v$}] at (.5,-3) {$ $};
             \draw[-, thick] (0,0) -- (-1,-1);
             \draw[-, thick] (0,0) -- (1,-1);
             \draw[-, thick] (-1,-1) -- (-1.5,-2);
             \draw[-, thick] (-1,-1) -- (-.5,-2);
             \draw[-, thick] (1,-1) -- (1.5,-2);
             \draw[-, thick] (1,-1) -- (.5,-2);
             \draw[-, thick] (.5,-2) -- (.5,-3);
    
            \end{tikzpicture}
            &
        \begin{tikzpicture}[shorten >=8pt, shorten <=8pt, scale=0.99]
            \node[circle, draw=black, fill=pmu] at (0,0) {$ $}; 
            \node[circle, draw=black, fill=white] at (-1,-1) {$ $}; 
            \node[circle, draw=black, fill=white,label=left:{\small $u$}] at (1,-1) {$ $}; 
            \node[circle, draw=black, fill=white] at (-1.5,-2) {$ $}; 
            \node[circle, draw=black, fill=white] at (-.5,-2) {$ $}; 
            \node[circle, draw=black, fill=white] at (1.5,-2) {$ $}; 
            \node[circle, draw=black, fill=pmu,label=left:{\small $v$}] at (.5,-2) {$ $}; 
            
             \draw[-, thick] (0,0) -- (-1,-1);
             \draw[-, thick] (0,0) -- (1,-1);
             \draw[-, thick] (-1,-1) -- (-1.5,-2);
             \draw[-, thick] (-1,-1) -- (-.5,-2);
             \draw[-, thick] (1,-1) -- (1.5,-2);
        \end{tikzpicture}
        \\
        A tree $T$ with $\psdthx(T) = 3$. & $T_{uv}$ with $\psdthx(T_{uv}) = 4$. & $T-uv$ with $\psdthx(T-{uv}) = 6$.

    \end{tabular}
    \caption{A tree $T$ with $7$ vertices, its subdivision, and an edge deletion.}\label{fig:operation}
\end{figure}
\pagebreak

%% file: special_graphs.tex
\section{Special Graphs}\label{sec:specialgraphs} 
We now determine the initial cost PSD product throttling number for serveral families of graphs. These results are summarzed in Table \ref{table:graphfamilies}.

In order to find the initial cost PSD product throttling number for cycles, we need some arithmetic involving ceilings and floors presented in \thref{lem:cycle_floors} and \thref{cycle_floors2}.
\begin{lem}\thlabel{lem:cycle_floors}
Let $k\geq 4$ and $n$ a positive integer. Then 
$$k \lp   1+\lc \frac{n-k}{2k}\rc  \rp   \geq 2+2\lc \frac{n-2}{4}\rc.$$
\end{lem}
\begin{proof}
We first observe that 
\begin{align*}
    k \lp   1+\lc \frac{n-k}{2k}\rc  \rp   & \geq  \lp  k+\lc \frac{n-k}{2}\rc  \rp   
    \geq k+ \frac{n-k}{2}  
    = \frac{k}{2} + \frac{n}{2}. 
\end{align*}
If $n = 2i$ for $i\geq 1$, then 
\begin{align*}
    \frac{k}{2} + \frac{n}{2}
     \geq 2+ i
     \geq 2+2\lc \frac{i-1}{2}\rc 
     = 2+2\lc \frac{2i-2}{4}\rc 
     = 2+2\lc \frac{n-2}{4}\rc.
\end{align*}

Now, if $n= 2i+1$ for $i\geq1$, and $k = 4$, through a straightforward case analysis we can see that
$$ 4 \lp   1+\lc \frac{n-4}{8}\rc  \rp   \geq 2+2\lc \frac{n-2}{4}\rc .$$
Suppose next, that $k\geq 5$ and $n = 2i+1$ then we have the following:
\begin{align*}
     \frac{k}{2} + \frac{n}{2}
     \geq \frac{5}{2} + \frac{n}{2}
     = 2+ i + 1
     \geq 2+2\lc \frac{i-\frac{1}{2}}{2}\rc 
     = 2+2\lc \frac{n-2}{4}\rc. \hspace{2cm}  \qedhere
\end{align*}
\end{proof}

\begin{lem}\thlabel{cycle_floors2}
Let $n$ be a positive integer. Then
   $$3 \lp   1+\lc \frac{n-3}{6}\rc  \rp  < 
   2+2\lc \frac{n-2}{4}\rc \text{ for $n\equiv 3 \mod 12$ }$$
   and
   $$3 \lp   1+\lc \frac{n-3}{6}\rc  \rp \geq 2+2\lc \frac{n-2}{4}\rc \text{ for $n\not\equiv 3 \mod 12$}.$$
\end{lem}
\begin{proof}
Let $n = 12i +j$ for $0\leq j \leq 11$. We first simplify our expressions. Observe that
\begin{align*}
     3 \lp 1+\lc\frac{n-3}{6}\rc \rp   = 3 \lp   1+\lc \frac{12i+j-3}{6}\rc  \rp   = 3 +3\lc  2i + \frac{j-3}{6}\rc = 3 +6i + 3\lc \frac{j-3}{6}\rc
\end{align*}
and that
\begin{align*}
    2 \lp   1+\lc \frac{n-2}{4}\rc  \rp   = 2 \lp   1+\lc \frac{12i+j-2}{4}\rc  \rp   = 2 +2\lc  3i + \frac{j-2}{4}\rc = 2 +6i +2\lc \frac{j-2}{4}\rc .
\end{align*}

Suppose $0\leq j\leq 2$. Then we have that $-1 < \frac{j-3}{6}<0$ and $-1 < \frac{j-2}{4}\leq 0$. And so,
\begin{align*}
   3 +6i + 3\lc \frac{j-3}{6}\rc  &= 3 + 6i \geq 2 +6i = 2 +6i +2\lc \frac{j-2}{4}\rc .
\end{align*}

Next, suppose $j = 3$, then
\begin{align*}
   3 +6i + 3\lc \frac{j-3}{6}\rc    =  3 +6i
   < 2 +6i +2
  =2 +6i +2\lc \frac{j-2}{4}\rc.
\end{align*}

 Next, let $4\leq j\leq 6$. Then we have that $0 < \frac{j-3}{6}< 1$ and $0 < \frac{j-2}{4}\leq 1$. And so, 
    \begin{align*}
        3 +6i + 3\lc \frac{j-3}{6}\rc    =
        3+6i +3 > 2+6i +2 = 2 +6i +2\lc \frac{j-2}{4}\rc.
    \end{align*}

Next, let $7\leq j\leq 10$. Then we have that $0 < \frac{j-3}{6}< 2$ and $1 < \frac{j-2}{4}\leq 2$. And so, 
    \begin{align*}
        3 +6i + 3\lc \frac{j-3}{6}\rc  \geq 3 + 6i +3 = 2 +6i +4 = 2 +6i +2\lc \frac{j-2}{4}\rc.
    \end{align*}
     Finally, let $j = 11$. Then we have that $1 < \frac{j-3}{6}< 2$ and $2 < \frac{j-2}{4}< 3$. And so, 
    \begin{align*}
        3 +6i + 3\lc \frac{j-3}{6}\rc    = 3 + 6i +6 > 2 +6i +6 = 2 +6i +2\lc \frac{j-2}{4}\rc. \hspace{2cm} \qedhere
    \end{align*}
\end{proof}

We are now ready to establish the value of $\psdthx \lp  C_n \rp$. 
\begin{thm}\thlabel{thm:cycles} Let $n\geq 4$. The initial value positive semidefinite product throttling number of a cycle $C_n$ is 
$$\psdthx\lp   C_n \rp  = \begin{cases}
3 \lp   1+\lc \frac{n-3}{6}\rc  \rp   & n = 12i+3, i\geq1 \\
2 \lp   1+\lc \frac{n-2}{4}\rc  \rp   & \text{ otherwise.}
\end{cases}$$
\end{thm}
\begin{proof}
We first show that $\psdthx\lp   C_n,k \rp  = k \lp   1+\lc \frac{n-k}{2k}\rc  \rp  $. To see this, let $S$ be the initial set of $k$ vertices, and observe that $C_n-S$ is a union of disjoint paths. For $\lc\frac{n}{3}\rc < k< n$ it is immediate that $\psdthx\lp C_n,k \rp >\psdthx\lp C_n,\lc\frac{n}{3}\rc \rp$. So assume $k \leq \lc\frac{n}{3}\rc$.
In the propagation step the endpoints of each path are adjacent in $C_n$ to a different vertex in $S$. Therefore, $S^{(1)}$ will contain the endpoints of each path. To color each path blue, we traverse the path from the endpoints, which gives us a propagation of $ \lc \frac{\ell}{2}\rc$ for each path, where $\ell$ is the order of each path. Therefore, the propagation time will be determined by the largest $\ell$. Placing the original set $S$ so as to minimize longest path, we conclude that $\psdpt \lp   C_n, k \rp   = \lc \frac{n-k}{2k}\rc$.
Since $\psdz \lp  C_n  \rp   \geq 2$, we have reduced the problem to finding $k\geq 2$ that minimizes $k \lp   1+\lc \frac{n-k}{2k}\rc  \rp  $.

By \thref{lem:cycle_floors} and \thref{cycle_floors2}, we see that if $k\geq 4$ or $k =3$ or $n \not \equiv 3 \mod 12$, then
$$ k \lp   1+\lc \frac{n-k}{2k}\rc  \rp   \geq 2+2\lc \frac{n-2}{4}\rc.$$ 

If $n \equiv 3 \mod 12$, then $$3 \lp   1+\lc \frac{n-3}{6}\rc  \rp < 2 \lp   1+\lc \frac{n-2}{4}\rc  \rp  \leq k\lp   1+\lc \frac{n-k}{2k}\rc  \rp$$ for $k \geq 4.$ 
\end{proof}

Using \thref{prop:cops_tree_cycle} and \thref{thm:cycles}, we have the next corollary.
\begin{cor}
Let $n\geq 4$. The cop product throttling number of cycles is
$$\operatorname{th}_c^\times \lp   C_n \rp  = \begin{cases}
3 \lp   1+\lc \frac{n-3}{6}\rc  \rp   & n = 12i+3, i\geq1 \\
2 \lp   1+\lc \frac{n-2}{4}\rc  \rp   & \text{ otherwise}.
\end{cases}$$
\end{cor}



Let $G$ be a graph. The \textit{complement} of a $G$ is denoted $\overline{G}$, and is defined by $V \lp \overline{G} \rp   = V \lp G \rp  $ and $E \lp \overline{G} \rp   = \{uv: uv \not \in E \lp G \rp  \}$. In \cite{book}, Hogben, Lin, and Shader show that when $n\geq 5$, $\psdz \lp \overline{C_n} \rp   = n-3, \psdpt \lp \overline{C_n} \rp   = 2$ if $n \neq 6$  and $ \psdpt\lp \overline{C_6} \rp   = 1$, $\psdz \lp \overline{P_n} \rp   = n-3$, and $\psdpt \lp \overline{P_n} \rp   = 2$. We use this knowledge for the next proposition.

\begin{prop} Let $n\geq 5$. Then
    \begin{enumerate}
     \item $\psdths \lp \overline{C_n} \rp   = \begin{cases} n-3 & \text{ if $n =6$}  \\ n-2 & \text{ if $n \neq 6$} \end{cases}$
     \item $\psdths \lp \overline{P_n} \rp   = n-2$.
\end{enumerate}
\end{prop}
\begin{proof}
     For the first statement, since $\psdpt \lp \overline{C_6} \rp   = 1$ and $\psdz \lp \overline{C_6} \rp   = 3$, $\psdths \lp \overline{C_6} \rp   = 3$. If $n \neq 6$, then $\psdpt \lp \overline{C_n} \rp   = 2$ and $\psdz \lp \overline{C_n} \rp   = n-3$. Let $V \lp \overline{C_n} \rp   = \{v_1,\ldots,v_n\}$ and $E \lp \overline{C_n} \rp   = \{v_iv_j: |i-j| \not \equiv  1 \mod n \}\setminus\{v_1v_n\}.$ Let $S = \{v_1,\ldots, v_{n-2}\}$. Then $\psdpts{\overline{C_n}}{S} = 1$.  \qedhere 
        
   
\end{proof}

Table \ref{table:graphfamilies} contains values for the initial value positive semidefinite product throttling number and non-initial value positive semidefinite product throttling for specific families of graphs determined by theorems that we have presented thus far. In the table, the symbol ? represents that the value is currently unkown for the parameter. 

\input{graph_table}

\clearpage

%% file: graph_table.tex
\begin{table}[ht]
{
    
    \begin{tabular}{|| c | c | c | c | c ||}
        \hline
         $G$   &$\psdz(G)$ &$\psdpt(G)$ &$\psdthx(G)$& $\psdths(G)$ \\
         \hline \hline
         & & & & \\
         $K_n$ & $n-1$ & $1$ & $n$ & $n-1$ \\
         & & & & \\
         \hline 
         & & & & \\
         $C_n$ & $2$ &$\lc\frac{n-2}{2} \rc$ & $\begin{cases}
            3\lp  1+\lc \frac{n-3}{6}\rc \rp & \scriptstyle{ n = 12i+3, i\geq 1} \\
            2\lp  1+\lc \frac{n-2}{4}\rc \rp & \scriptstyle{\text{ otherwise.} }
            \end{cases}$ & $\lc\frac{n}{3}\rc$\\
         & & & & \\
          \hline 
          & & & & \\
          $K_{s,t}$ &$\min(s,t)$ & $1$ & $\min\{2t, 2s\}$ &$\min\{s,t\}$ \\
         
          & & & & \\
          \hline 
          & & & & \\
          Any tree $T$ & $1$ &$\rad(T)$ &$1 + \rad(T)$  & ?\\
          & & & & \\
          \hline 
          & & & & \\
          $P_n$ & $1$ &$\lc\frac{n-1}{2}\rc$ &$1 +\lc\frac{n-1}{2}\rc$  & $\lc\frac{n}{3}\rc$\\
          & & & & \\
          \hline 
          & & & & \\
          $Q^d$ & $2^{d-1}$ & $1$ &  $2^d$ & $2^{d-1}$\\
          & & & & \\
          \hline 
          & & & & \\
           $K_{n_1,\ldots,n_k}$ & & & & \\ {$\scriptstyle{n_1<\dots <n_k}$} & $n_1 + \cdots + n_{k-1}$&$1$  & $\min\{n,2(n_1+\dots +n_{k-1})\}$  & $n_1+\dots +n_{k-1}$\\
           & & & & \\
           \hline
          & & & & \\
          $\overline{C_n}, n\geq 5$ & $n-3$ & $\begin{cases} 2 & \scriptstyle{\text{ if }  n\neq 6} \\ 1  & \scriptstyle{\text{ if } n= 6 }\end{cases}$ & $n$ &  $\begin{cases}  n-2 & \scriptstyle{\text{ if } n \neq 6} \\ n-3 & \scriptstyle{\text{ if } n =6 } \end{cases}$ \\
          & & & & \\
          \hline
          & & & & \\
          $\overline{P_n}, n\geq 5$ & $n-3$ & $2$ & $n$ &  $n-2$\\
          & & & & \\
          \hline
    \end{tabular}
}
\caption{Graph Families. Results for the graph families that are not proved here can be found in \cite{book}.}\label{table:graphfamilies}
\end{table}